\documentclass[11pt,reqno]{amsart}
\usepackage{amsfonts, amsmath, amssymb, amscd, amsthm, bm}
\usepackage{amsrefs}

\usepackage{enumerate}
\allowdisplaybreaks[4]
\usepackage{cases}
\usepackage{url}
\usepackage[linktocpage=true,colorlinks,citecolor=magenta,linkcolor=blue,urlcolor=magenta]{hyperref}

\usepackage{multicol}
\usepackage{comment}
\usepackage[title]{appendix}

\newtheorem{thm}{Theorem}[section]
\newtheorem{cor}[thm]{Corollary}

\newtheorem{lem}[thm]{Lemma}
\newtheorem{prop}[thm]{Proposition}
\theoremstyle{definition}
\newtheorem{defn}[thm]{Definition}
\theoremstyle{theorem}
\newtheorem{rem}[thm]{Remark}

\theoremstyle{claim}

\numberwithin{equation}{section}
\def\R{\mathbb{R}^{n}(c)}
\def\H{\mathbb{H}^{n}(-1)}
\def\S{\mathbb{S}^{n}}
\def\et{\tilde{e}}
\def\th{\tilde{h}}
\def\tH{\tilde{H}}
\def\tg{\tilde{g}}
\def\tom{\tilde{\omega}}
\def\tt{\tilde{\theta}}
\def\tn{\widetilde{\nabla}}
\def\tR{\tilde{R}}

\def\la{\lambda}

\def\l{\langle}
\def\r{\rangle}

\DeclareMathOperator\tr{tr}
\DeclareMathOperator\arsinh{arsinh}
\DeclareMathOperator\dive{div}
\DeclareMathOperator\genus{genus}

\def\confs{\mathrm{Conf}(\mathbb{S}^n)}
\def\confb{\mathrm{Conf}(B^n)}

\begin{document}

\title[Conformal and extrinsic upper bounds for eigenvalues]{Conformal and extrinsic upper bounds for the harmonic mean of Neumann and Steklov eigenvalues}

\author[H. Chen]{Hang Chen}
\address{School of Mathematics and Statistics, Northwestern Polytechnical University, 1 Dongxiang Road, Xi'an, 710129, P.R.China}
\email{\href{mailto:chenhang86@nwpu.edu.cn}{chenhang86@nwpu.edu.cn}}

\begin{abstract}
Let $M$ be an $m$-dimensional compact Riemannian manifold with boundary.
We obtain the upper bounds of the harmonic mean of the first $m$ nonzero Neumann eigenvalues and Steklov eigenvalues involving the conformal volume and relative conformal volume, respectively.
We also give an optimal sharp extrinsic upper bound for closed submanifolds in space forms.
These extend the previous related results for the first nonzero eigenvalues.
\end{abstract}



\maketitle

\section{Introduction}
The study of eigenvalues on manifolds is an important topic in Riemannian geometry, and there have been lots of results, including various estimates bounded from below or above in different situations.
In this paper, we study the upper bounds of the harmonic mean of the lower Neumann and Steklov eigenvalues.

\subsection{Conformal upper bounds for Neumann eigenvalues}
Let $(M,g)$ be an $m$-dimensional compact Riemannian manifold (possibly with boundary).
We always assume $m\ge 2$ throughout this article except special declaration.
It is well known that the eigenvalues of Laplacian on $M$ (with Neumann boundary condition if $\partial M\neq \emptyset$) are discrete and satisfy
\begin{equation*}
	0=\lambda_0<\lambda_1\le \lambda_2\le \cdots \to +\infty.
\end{equation*}

In 1982, Li and Yau \cite{LY82} introduced the concept of the conformal volume in order to study the Willmore conjecture and the first eigenvalue of compact surfaces.
\begin{defn}[cf. \cite{LY82}]
	Let $M$ be an $m$-dimensional compact Riemannian manifold with boundary that admits a conformal map $\phi: M\to \mathbb{S}^n$.
	Define
\begin{align}
	V_c(M, n, \phi)&:=\sup_{\gamma\in \confs} V (\gamma (\phi(M))),\nonumber\\
	V_c(M, n)&:=\inf_{\phi}V_c(M, n, \phi),\nonumber\\
	V_c(M)&:=\lim_{n\to +\infty}V_c(M,n).\label{eq-cv}
\end{align}
Here $\confs$ is the group of conformal diffeomorphisms of $\mathbb{S}^n$, and the infimum is over all non-degenerate conformal maps $\phi: M\to \mathbb{S}^n$.

We call $V_c(M, n, \phi), V_c(M, n)$ and $V_c(M)$ the $n$-conformal volume of $\phi$, the $n$-conformal volume of $M$, and the conformal volume $M$, respectively.
\end{defn}

We remark that, the conformal map $\phi$ exists for $n$ big enough due to the Nash embedding theorem (via the stereographic projection), and the limit in \eqref{eq-cv} exists since $V_c(M, n)\ge V_c(M, n+1)$ (cf. \cite{LY82}*{Fact 4}).

Let us set the following condition for convenience.
\begin{itemize}
	\item[\textbf{(C1)}] $(M, g)$ admits, up to a homothety, a minimal isometric immersion in $\mathbb{S}^n$ given by a subspace of the first eigenspace.
\end{itemize}

The first eigenvalue $\lambda_1$ can be bounded by the conformal volume from above. Precisely, we have

\begin{thm}[\cites{LY82, ESI86}]\label{thm-conf}
	Let	$(M,g)$ be an $m$-dimensional compact Riemannian manifold. Then the first nonzero eigenvalue $\lambda_1$ of the Laplacian (with Neumann boundary condition if $\partial M\neq \emptyset$) satisfies
		\begin{equation*}
			\lambda_1\le m\Bigl(\frac{V_c(M,n)}{V(M,g)}\Bigr)^{2/m}
		\end{equation*}
		for all $n$ for which $V_{c}(M, n)$ is defined.

		Equality implies (C1).
		If $M$ is closed, then (C1) implies equality as well.
\end{thm}

This estimate was first obtained by Li and Yau \cite{LY82} for $m=2$, and then by El Soufi and Ilias \cite{ESI86} for general $m\ge 3$. Matei \cite{Mat13} generalized it to the $p$-Laplacian.
For higher eigenvalues of a closed manifold, Kokarev \cite{Kok20} recently proved that there is a constant $C(n,m)$ depending only on $m$ and $n$ such that
\begin{equation}\label{Kok}
	\lambda_k\le C(n,m)\Bigl(\frac{V_c(M,n)}{V(M,g)}\Bigr)^{2/m} k^{2/m}
\end{equation}
holds for any $k\ge 1$, which is compatible with the famous Weyl asymptotic formula:
\begin{equation*}
	\lambda_k V(M,g)^{2/m}\sim \frac{4\pi^2}{\omega_m^{2/m}}k^{2/m} \mbox{ as } k\to \infty,
\end{equation*}
where $\omega_m$ is the volume of the unit ball in $\mathbb{R}^m$.
This can be viewed as an improvement of Korevaar's result \cite{Kor93}, which says that
\begin{equation*}
	\lambda_k\le Ck^{2/m}
\end{equation*}
for some constant $C$ depending on $[g]$ in a rather implicit way.

In this paper, we consider the harmonic mean of the first $m$ nonzero eigenvalues.
We write the harmonic mean of $k$ positive numbers $a_1,\cdots, a_k$ as
\begin{equation*}
	\mathfrak{H}(a_1,\cdots, a_k)=\Bigl(\frac{a_1^{-1}+\cdots+a_k^{-1}}{k}\Bigr)^{-1}.
\end{equation*}
If  $a_1\le \cdots \le a_k$, then one can easily check that
\begin{equation}\label{eq-1.4}
	a_1=\mathfrak{H}(a_1)\le \mathfrak{H}(a_1,a_2)\le \cdots \le \mathfrak{H}(a_1,\cdots, a_{k-1})\le \mathfrak{H}(a_1,\cdots, a_k).
\end{equation}

The first main result in this paper is the following estimate involving the conformal volume for $\mathfrak{H}(\lambda_1,\cdots,\lambda_m)$.

\begin{thm}\label{thm-conf-h}
	Let	$M$ be an $m$-dimensional compact Riemannian manifold. Then the first $m$ nonzero eigenvalues of the Laplacian  (with Neumann boundary condition if $\partial M\neq \emptyset$) satisfy
		\begin{equation}\label{eq-conf-h}
			\mathfrak{H}(\lambda_1,\cdots,\lambda_m)\le m\Bigl(\frac{V_c(M,n)}{V(M,g)}\Bigr)^{2/m}
		\end{equation}
		for all $n$ for which $V_{c}(M, n)$ is defined.

		Equality implies $\lambda_1=\dots=\lambda_m$ and (C1).
		If $M$ is closed, then (C1) implies equality as well.
\end{thm}

\begin{rem}
	For some special manifolds, such as compact symmetric spaces of rank 1 and the minimal Clifford torus with their canonical metric $g_{\mathrm{can}}$, we have $V_c(M, n)=V(M,\frac{\lambda_1}{m} g_{\mathrm{can}})$ for $n+1$ greater or equal to the multiplicity of $\lambda_1$ (cf. \cite{ESI86}*{pp.~266--267}).
	In these cases, the upper bound is explicit.
\end{rem}

Clearly, Theorem \ref{thm-conf-h} is stronger than Theorem \ref{thm-conf} due to \eqref{eq-1.4}.
For $m=2$, the upper bound can be controlled by the topology of the surface.
\begin{cor}\label{cor2.1}
	Let	$(M,g)$ be a closed surface\footnote{In this article, we always consider orientable surfaces except special declaration.}. 
	Then
	\begin{equation}
		\mathfrak{H}(\lambda_1,\lambda_2)\le
			\frac{8\pi}{V(M,g)} \Big[\frac{\genus(M)+3}{2}\Big],
			\label{eq-1.16}
	\end{equation}
	where $[\,\cdot\,]$ denotes the integer part.  In particular, for a 2-sphere we have
	\begin{equation}
		\mathfrak{H}(\lambda_1,\lambda_2)\le
			\frac{8\pi}{V(\mathbb{S}^2,g)},
			\label{eq-1.16-s}
	\end{equation}
\end{cor}
This corollary is directly obtained from Theorem \ref{thm-conf-h} since
Li and Yau \cite{LY82} pointed out that
\begin{equation}\label{eq-1.17}
	V_c(M,g)\le 4\pi (\genus(M)+1),
\end{equation}
and then El Soufi and Ilias \cite{ESI84} remarked that the factor $\genus(M)+1$ in \eqref{eq-1.17} can be improved to $\big[\frac{\genus(M)+3}{2}\big]$,

Let's review some related results for $m=2$. Obviously, Corollary \ref{cor2.1} implies
\begin{equation}\label{eq-1.18}
	\lambda_1\le \frac{8\pi}{V(M,g)}(\genus(M)+1),
\end{equation}
which was originally proved by Hersch \cite{Her70} for $M=\mathbb{S}^2$ (i.e., $\genus(M)=0$) and Yang and Yau \cite{YY80} for $\genus(M)\ge 1$. Actually, Yang and Yau \cite{YY80} proved a stronger result
\begin{equation*}
	\mathfrak{H}(\lambda_1,\lambda_2, \lambda_3)\le
	\frac{8\pi}{V(M,g)}(\genus(M)+1).
\end{equation*}
By \eqref{eq-1.17}, Li and Yau \cite{LY82} gave a simpler proof of \eqref{eq-1.18}; they also proved that (\cite{LY82}*{Corollary 2})
\begin{equation*}
	\mathfrak{H}(\lambda_1,\cdots, \lambda_k)
	\le n\frac{V_c(M, n-1)}{V(M)} \mbox{ for any $n\ge k$},
\end{equation*}
but the constant is not sharp.

Another remarkable estimate for higher eigenvalues of a 2-sphere $(\mathbb{S}^2,g)$ is
\begin{equation}\label{eq-Nad}
	\lambda_k(\mathbb{S}^2,g) \le \frac{8\pi k}{V(\mathbb{S}^2,g)}, \mbox{  for any  }  k\ge 1.
\end{equation}
The metric $g$ is smooth outside a finite number of conical singularities.
For $k=1$, this is just Hersch's result \cite{Her70} and the equality in \eqref{eq-Nad} is attained if and only if $g$ is a round metric.
For $k\ge 2$, the inequality in \eqref{eq-Nad} is strict, and the equality can be  attained in the limit by a sequence of metrics degenerating to a union of $k$ touching identical round spheres.
This was proved by Nadirashvili \cite{Nad02} (and Petrides \cite{Pet14} with a different argument) for $k=2$, Nadirashvili and Sire \cite{NS17} for $k=3$, and Karpukhin et al. \cite{KNPP21} for any $k\ge 2$ recently.
From \eqref{eq-Nad} we have
\begin{equation*}
	\mathfrak{H}(\lambda_1,\lambda_2)\le
	\frac{4}{3}
	\frac{8\pi}{V(\mathbb{S}^2,g)},
\end{equation*}
which is weaker than \eqref{eq-1.16-s}.
Actually, \eqref{eq-conf-h} cannot be derived from \eqref{Kok}.
Intuitively, our results imply that if $\lambda_1$ is close to the upper bound, then $\lambda_m$ cannot be far away from the same upper bound.

\subsection{Conformal upper bounds for Steklov eigenvalues}
Besides Neumann eigenvalues, Steklov eigenvalues are intensively studied for a compact Riemannian manifold $(M, g)$ with nonempty boundary $\partial M$.
A real number $\sigma$ is called a Steklov eigenvalue if there exists a non-zero function $u$ (which is called the Steklov eigenfunction corresponding to $\sigma$) on $M$ satisfying
\begin{equation*}
	\begin{cases}
	\Delta u=0, & \mbox{on } M;\\
	\frac{\partial u}{\partial\nu} \,\,=\sigma u, & \mbox{on } \partial M,
	\end{cases}
	\end{equation*}
where $\nu$ is the outward normal on $\partial M$.

The Steklov problem has a physical background, and it can be traced back to the turn of the 20th century, when Steklov studied liquid sloshing \cites{Ste02, Ste02a}. The readers can refer to \cite{KKK14} for the history and a very recent survey \cite{CGGS22} for research in this area.

The Steklov eigenvalues can be interpreted as the eigenvalues of the Dirichlet-to-Neumann operator $L$, which is defined as follows:
\begin{equation*}
	\begin{aligned}
		L: C^{\infty}(\partial M)&\to C^{\infty}(\partial M)\\
	Lu&=\frac{\partial \hat{u}}{\partial \nu},
	\end{aligned}
\end{equation*}
where $u$ is a function on $\partial M$ and $\hat{u}$ is the harmonic extension of $u$, i.e.,
\begin{equation*}
	\begin{cases}
	\Delta \hat{u}=0, & \mbox{on } M;\\
	\,\,\,\,\, \hat{u} = u, & \mbox{on } \partial M.
	\end{cases}
	\end{equation*}
$L$ is a nonnegative self-adjoint operator with discrete spectrum
\begin{equation*}
	0=\sigma_0<\sigma_1\le \sigma_2\le \cdots \to +\infty.
\end{equation*}

Fraser and Schoen \cite{FS11} gave an upper bound involving the \emph{relative conformal volume} for the first nonzero Steklov eigenvalue, which is analogous to Theorem \ref{thm-conf} for the first Neumann eigenvalue of the Laplacian.
\begin{thm}[\cite{FS11}*{Theorem 6.2}]\label{thm-FS}
	Let $(M, g)$ be a compact $m$-dimensional Riemannian manifold with nonempty boundary. Then the first nonzero Steklov eigenvalue $\sigma_1$ satisfies
	\begin{equation*}
		\sigma_1 V(\partial M)V(M)^{\frac{2-m}{m}}	\le m V_{rc}(M, n)^{2/m}
	\end{equation*}
	for all $n$ for which $V_{rc}(M, n)$ is defined.

	Equality implies (C2).
\end{thm}

Let us explain the condition (C2) and the notation $V_{rc}(M,n)$ in the above theorem.
\begin{itemize}
	\item[\textbf{(C2)}] There exists a conformal harmonic map $\phi: M\to B^n$ which (after rescaling the metric $g$) is an isometry on $\partial M$, with $\phi(\partial M)\subset \partial B^n$ and such that $\phi(M)$ meets $\partial B^n$ orthogonally along $\phi(\partial M)$.
	For $m>2$,  this map is an isometric minimal immersion of $M$ to its image. Moreover, the immersion is given by a subspace of the first eigenspace.
\end{itemize}
\begin{defn}[cf. \cite{FS11}]
	Let $M$ be an $m$-dimensional compact Riemannian manifold with boundary that admits a conformal map $\phi: M\to B^n$ with $\phi(\partial M)\subset \partial B^n$. Define
	\begin{align}
		V_{rc}(M, n, \phi)&:= \sup_{\gamma\in \confb} V (\gamma (\phi(M))),\nonumber\\
		V_{rc}(M, n)&:= \inf_{\phi} V_{rc}(M, n, \phi),\nonumber\\
		V_{rc}(M)&:=\lim_{n\to +\infty} V_{rc}(M,n).\label{eq1.12}
	\end{align}
	Here $\confb$ is the group of conformal diffeomorphisms of $B^n$, and the infimum is over all non-degenerate conformal maps $\phi: M\to B^n$ with $\phi(\partial M)\subset \partial B^n$.

	We call $V_{rc}(M, n, \phi), V_{rc}(M, n)$ and $V_{rc}(M)$ the relative $n$-conformal volume of $\phi$, the relative $n$-conformal volume of $M$, and the relative conformal volume of $M$, respectively.
\end{defn}

Similar to the conformal volume, the existence of the limit in \eqref{eq1.12} is from the monotonicity of $V_{rc}(M, n)$ (cf. \cite{FS11}*{Lemma 5.7}).

We extend Theorem \ref{thm-FS} to the harmonic mean of the first $m$ nonzero eigenvalues.

\begin{thm}\label{thm-FS-h}
	Let $(M, g)$ be a compact $m$-dimensional Riemannian manifold with nonempty boundary.
	Then the first $m$ nonzero Steklov eigenvalues satisfy
	\begin{equation}\label{eq-FS-h}
		\mathfrak{H}(\sigma_1,\cdots,\sigma_m) V(\partial M)V(M)^{\frac{2-m}{m}}	\le m V_{rc}(M, n)^{2/m}
	\end{equation}
	for all $n$ for which $V_{rc}(M, n)$ is defined.

	Equality implies $\sigma_1=\dots=\sigma_m$ and (C2).
\end{thm}

For compact surfaces, we have
\begin{thm}\label{thm-FS-h-2}
	Let $(M, g)$ be a compact surface with $k$ boundary components. Then
	\begin{equation*}
		\mathfrak{H}(\sigma_1,\sigma_2) V(\partial M)	\le 2 (\genus(M)+k)\pi.
	\end{equation*}

	In particular, for a compact simply-connected surface with boundary, we have
	\begin{equation*}
		\mathfrak{H}(\sigma_1,\sigma_2) V(\partial M)	\le 2\pi,
	\end{equation*}
	Moreover, equality holds if and only if there is a conformal map from $M$ to the unit disk which is an isometry on the boundary.
\end{thm}
This theorem extends an estimate of Fraser-Schoen (\cite{FS11}*{Theorem 2.3}) saying
\begin{equation}\label{eq-FS-2}
	\sigma_1 V(\partial M)	\le 2 (\genus(M)+k)\pi
\end{equation}
under the same assumptions.
When $\genus(M)=0$, \eqref{eq-FS-2} was obtained by Weinstock \cite{Wei54} for $k=1$ and it is sharp; however, it is not sharp for $k\ge 2$ (see \cite{FS11}*{Theorem 2.3}).
These estimates for Steklov eigenvalues are analogous to the estimates \eqref{eq-1.16} and \eqref{eq-1.18} for eigenvalues of the Laplacian.
We also have
\begin{thm}\label{thm-FS-h-3}
	Let $(M, g)$ be a compact surface with boundary. Then
	\begin{equation*}
		\mathfrak{H}(\sigma_1,\sigma_2) V(\partial M)	<8\pi(\genus(M)+1).
	\end{equation*}

	In particular, for a compact surface with boundary of genus $0$, we have
	\begin{equation*}
		\mathfrak{H}(\sigma_1,\sigma_2) V(\partial M)	< 8\pi.
	\end{equation*}
\end{thm}
This extends some corresponding estimates for $\sigma_1$ (cf. \cite{Kok14}*{Theorem $A_1$} and \cite{CGGS22}*{Theorem 3.6}).

\subsection{Reilly inequalities in space forms}
The conformal upper bounds are proved by finding suitable test functions via conformal maps.
This technique also allows us to obtain some extrinsic upper bounds for closed submanifolds in space forms.

Let $(\R, h_c)$ be the $n$-dimensional simply-connected space form of constant curvature $c$ equipped with the canonical metric $h_c$, namely, $\R$ represents the Euclidean space $\mathbb{R}^{n}$, the hyperbolic space $\H$ and the unit sphere $\S$ for $c=0, -1$ and $1$, respectively. When $M$ is a closed submanifold of $\R$,  the first nonzero eigenvalue can be bounded by the mean curvature $\mathbf{H}$ from above.
\begin{thm}[\cites{Rei77, ESI92}]\label{thm-Rei}
	Let	$M$ be an $m$-dimensional closed submanifold in $\R$. Then the first nonzero eigenvalue of Laplacian satisfies
		\begin{equation}\label{eq-Rei}
			\lambda_1\le\frac{m}{V(M)}\int_M(c+|\mathbf{H}|^2).
		\end{equation}
		Moreover, equality in \eqref{eq-Rei} holds if and only if either

		(1) $M$ is minimal in $\mathbb{S}^n$ and the immersion is given by a subspace of the first eigenspace, or

		(2) $M$ is minimal in a geodesic sphere $\Sigma_c$ of $\R$, where the geodesic radius $r_c$ of $\Sigma_c$ is given by
		\begin{equation*}
			r_0=\left(\frac{m}{\lambda_1}\right)^{1/2},\quad r_1=\arcsin r_0,\quad r_{-1}=\arsinh r_0.
		\end{equation*}
\end{thm}

\begin{rem}
	Case (1) of Theorem \ref{thm-Rei} occurs only when $c=1$. For Case (2) of Theorem \ref{thm-Rei}, we use the following convention:

	When $M$ is a closed hypersurface of $\mathbb{R}^n(c)$ (i.e., $n-m=1$), ``$M$ lies in a geodesic sphere $\Sigma_c$ of $\mathbb{R}^n(c)$'' means ``$M=\Sigma_c$''.
\end{rem}

Theorem \ref{thm-Rei} is usually referred ``Reilly inequality'' as it was first proved by Reilly \cite{Rei77} for $c=0$ by using the coordinate functions as the test functions after translating the center of mass of $M$ to the origin.
By embedding the sphere $\mathbb{S}^n\to \mathbb{R}^{n+1}$, one can reduce the case $c=1$ to the case $c=0$.
But this method fails for $c=-1$.  The case $c=-1$ was proved by El Soufi and Ilias \cite{ESI92} via the technique of conformal transformations.

Kokarev \cite{Kok20} also proved a version of Reilly inequality for higher eigenvalues, that is,
\begin{equation*}
	\lambda_k\le C(n,m)\Bigl(\frac{1}{V(M)}\int_M(c+|\mathbf{H}|^2)\Bigr)k
\end{equation*}
for a closed submanifold $M^m$ of $\R$.
In this paper, we prove
\begin{thm}\label{thm1.1}
	Let $M$ be an $m$-dimensional closed submanifold in $\R$. Then the first $m$ nonzero eigenvalues of the Laplacian satisfy
	\begin{equation}\label{eq_thm1}
		\mathfrak{H}(\la_1,\cdots, \la_{m}) \le \frac{m}{V(M)}\int_M\left(c+|\mathbf{H}|^2\right).
	\end{equation}
	Moreover, equality in \eqref{eq_thm1} holds if and only if either

	(1) $M$ is minimal in $\mathbb{S}^n$ and the immersion is given by a subspace of the first eigenspace, or

	(2) $M$ is minimal in a geodesic sphere $\Sigma_c$ of $\R$, where the geodesic radius $r_c$ of $\Sigma_c$ is given by
	\begin{equation*}
		r_0=\left(\frac{m}{\lambda_1}\right)^{1/2},\quad r_1=\arcsin r_0,\quad r_{-1}=\arsinh r_0.
	\end{equation*}
\end{thm}

\subsection{Extrinsic estimates for submanifolds of the Euclidean space}\label{Sect1.4}
When $M$ is a submanifold of $\mathbb{R}^n$, we can directly use the coordinate functions as the test functions without conformal maps.
This is just the original idea in \cite{Rei77}, which allows us to get more general extrinsic upper bounds, including both closed eigenvalues (i.e., eigenvalues of the Laplacian for closed manifolds) and Steklov eigenvalues.
In fact, the original Reilly's inequalities in \cite{Rei77} involve not only the mean curvature, but also the higher order mean curvatures.
By this approach, we can prove the following results. First, considering the eigenvalues of the Laplacian, we have
\begin{thm}\label{thm-ext-c}
	Let	$M$ be an $m$-dimensional closed submanifold of $\mathbb{R}^n$.
	Let $T$ be a symmetric, divergence-free $(1, 1)$-tensors on $M$.
	Then the first $m$ nonzero eigenvalues of the Laplacian satisfy
	\begin{equation}\label{eq-ext-c}
		\mathfrak{H}(\lambda_1,\cdots,\lambda_m) \Bigl(\int_M\tr T \Bigr)^2 \le m V(M)\int_M |H_T|^2.
	\end{equation}
	Assume that $H_T\not\equiv 0$, then equality holds if and only if the following conditions hold simultaneously:

	(1) $\tr T$ is constant;

	(2) $M$ is minimal and $T$-minimal in a hypersphere of $\mathbb{R}^n$ with the radius $\sqrt{m/\lambda_1}$.
\end{thm}

By embedding the sphere $\mathbb{S}^n\to \mathbb{R}^{n+1}$, we have
\begin{thm}\label{thm-ext-c-sphere}
	Let	$M$ be an $m$-dimensional closed submanifold of $\mathbb{S}^n$.
	Let $T$ be a symmetric, divergence-free $(1, 1)$-tensors on $M$.
	Then the first $m$ nonzero eigenvalues of the Laplacian satisfy
	\begin{equation}\label{eq-ext-c-sphere}
		\mathfrak{H}(\lambda_1,\cdots,\lambda_m) \Bigl(\int_M\tr T \Bigr)^2 \le m V(M)\int_M \big(|H_T|^2+(\tr T)^2\big).
	\end{equation}
	Assume that $H_T\not\equiv 0$, then equality holds if and only if the following conditions hold simultaneously:

	(1)  $\tr T$ are constant;

	(2) either $M$ is minimal and $T$-minimal in $\mathbb{S}^n$ and the immersion is given by a subspace of the first eigenfunctions;
	or $M$ is minimal and $T$-minimal in a geodesic sphere of $\mathbb{S}^n$ with the geodesic radius $\arcsin (\sqrt{m/\lambda_1})$;
\end{thm}

We will introduce the definitions of $H_T$ and $T$-minimality in Sect.~\ref{Sect:2.3}. The above theorems imply \cite{Gro02}*{Proposition 2.1}.
If we take $T=T_r$, the Newton transformation of the second fundamental form, then $H_T$ and $\tr T$ are corresponding to the $(r+1)$-th and $r$-th mean curvature, respectively.
Hence, we immediately conclude
\begin{cor}\label{cor_Rei}
	Let	$M$ be an $m$-dimensional closed submanifold of $\mathbb{R}^n(c), c=0, 1$.

(1) If $n>m+1$, then for each even $r\in \{0,\cdots m-1\}$, we have
\begin{equation*}
\mathfrak{H}(\lambda_1,\cdots,\lambda_m)\Big(\int_M H_{r} \Big)^2\leq mV(M)\int_M(|\mathbf{H}_{r+1}|^2+cH_r^2).
\end{equation*}

(2) If $n=m+1$, then for each $r\in \{0,\cdots m-1\}$, we have
\begin{equation*}
\mathfrak{H}(\lambda_1,\cdots,\lambda_m)\Big(\int_M H_{r} \Big)^2\leq mV(M)\int_M(H_{r+1}^2+cH_r^2).
\end{equation*}
Here $\mathbf{H}_{r+1}$ is the $(r+1)$-th mean curvature vector of $M$.
\end{cor}

Corollary \ref{cor_Rei} implies \cite{Rei77}*{Theorem A} and \cite{Gro02}*{Theorem 2.1}. In particular, we again obtain Theorem \ref{thm1.1} for $c=0,1$ when taking $r=0$.

For the Steklov eigenvalues, we prove the following theorem, which recovers \cite{IM11}*{Theorem 1.2}.
\begin{thm}\label{thm-ext-s}
	Let	$M$ be an $m$-dimensional compact submanifold with non-empty boundary $\partial M$ of $\mathbb{R}^n$.
	Let $T$ be a symmetric, divergence-free $(1, 1)$-tensors on $\partial M$.
	Then the first $m$ nonzero Steklov eigenvalues satisfy
	\begin{equation}\label{eq-ext-s}
		\mathfrak{H}(\sigma_1,\cdots,\sigma_m) \Bigl(\int_{\partial M}\tr T \Bigr)^2 \le m V(M)\int_{\partial M}|H_T|^2.
	\end{equation}

	In particular, when $T=I$, we obtain
	\begin{equation*}
		\mathfrak{H}(\sigma_1,\cdots,\sigma_m)  \le \frac{m V(M)}{\bigl(V(\partial M)\bigr)^2}\int_{\partial M}|\mathbf{H}^{\partial M}|^2,
	\end{equation*}
	where $\mathbf{H}^{\partial M}$ denotes the mean curvature of $\partial M$ in $\mathbb{R}^n$.

	Assume that $H_T\not\equiv 0$, then we have

	(1) if $n>m$, then equality holds if and only if $x$ immerses $M$ minimally in $B^n(\frac{1}{\sigma_1})$ with $x(\partial M)\subset \partial B^n(\frac{1}{\sigma_1})$ such that $x(M)$ meets $\partial B^n(\frac{1}{\sigma_1})$ orthogonally along $x(\partial M)$ and $\partial M$ is $T$-minimal in $\partial B^n(\frac{1}{\sigma_1})$.

	(2) if $n=m$ and $M$ is a bounded domain of $\mathbb{R}^n$, then equality holds if and only if $M$ is a round ball and $tr T$ is constant.
\end{thm}

\medskip
This paper is organized as follows.
In Sect.~\ref{Sect-2} we recall some formulae about conformal maps and submanifolds; most notations and concepts involved in this article are introduced as well.
Then we give the proofs of Theorems \ref{thm-conf-h} and \ref{thm-FS-h} in Sect.~\ref{Sect-3} and Sect.~ \ref{Sect-4}, respectively.
In Sect.~\ref{Sect-5} we prove the Reilly-type inequality Theorem \ref{thm1.1}; a general version Theorem \ref{thm-3.6} is also obtained.
In the last section, we give the proofs of results in Sect.~\ref{Sect1.4}.
The choice of test functions is analogous in these proofs, but discussions on equalities are different.
It is worth pointing out that, unlike the estimates for the first eigenvalues, we need more careful analysis when studying the equality cases (cf. Sect.~\ref{Sect3.2} and Sect.~\ref{Sect5.2}).

\section{Preliminaries}\label{Sect-2}
In this section, we give the relations between some geometric quantities of $M^m$ as a submanifold of  $(N^{n}, g_N)$ and their corresponding quantities of $M$ as a submanifold in $(N^{n}, \tilde{g}_N)$, where $\tilde{g}_N$ is conformal to $g_N$.
Although the relations are well-known in the literature (cf. \cite{Che73a}), we provide a brief proof of these relations for the reader's convenience.
We also recall some basic formulae for submanifolds of space forms.

We use the following convention on the ranges of indices except special declaration:
\begin{equation*}
1\leq i, j, k, \ldots \leq m;\quad
m+1\leq\alpha, \beta, \gamma, \ldots \leq n;\quad
1\leq A, B, C, \ldots \leq n.
\end{equation*}

\subsection{Conformal relations: the ambient manifolds}
Let $(N, g_N)$ be an $n$-dimensional manifold equipped with a Riemannian metric $g_N$.
Let $\{e_A\}_{A=1}^{n}$ and $\{\omega_A\}_{A=1}^{n}$ be a local orthonormal frame and the dual coframe of $(N,g_N)$, respectively.
Then the structure equations of $(N,g_N)$ are (cf. \cite{Chern1968}):
\begin{equation}
\left\{\begin{aligned}\label{eq21}
& d\omega_{A}=\sum_{B}\omega_{AB}\wedge\omega_{B}, \\
&\omega_{AB}+\omega_{BA}=0,
\end{aligned}\right.
\end{equation}
where $\{\omega_{AB}\}$ are the connection $1$-forms of $(N,g_N)$.

Let $\tg_N$ be a metric on $N$ conformal to $g_N$. We assume $\tg_N=e^{2\rho}g_N$ for some  $\rho\in C^{\infty}(N)$.
Then $\{\et_A=e^{-\rho}e_A\}$ is a local orthonormal frame of $(N,\tg_N)$, and $\{\tom_A=e^{\rho}\omega_A\}$ is the dual coframe of  $\{\et_A\}$. The structure equations of $(N,\tg_N)$ are given by
\begin{equation}
\left\{\begin{aligned}\label{eq25}
&d\tom_{A}=\sum\limits_{B}\tom_{AB}\wedge\tom_{B},\\
&\tom_{AB}+\tom_{BA}=0,
\end{aligned}\right.
\end{equation}
where $\{\tom_{AB}\}$ are the connection $1$-forms of $(N,\tg_N)$.

From \eqref{eq21} and \eqref{eq25}, we can solve out
\begin{equation}\label{eq26}
\tom_{AB}=\omega_{AB}+\rho_A\omega_B-\rho_B\omega_A,
\end{equation}
where $\rho_A$ denotes the covariant derivative of $\rho$ with respect to $e_A$.

Let $\bar{\nabla}$ (or $\tilde{\nabla}$ resp.) denote Levi-Civita connection on $N$ with respect to $g_N$ (or $\tg_N$ resp.).
In general, for any $F\in C^{\infty}(N)$,
\begin{equation*}
	\sum_{A=1}^{n}\tn_{A}F\;\tom_A=\sum_{A=1}^{n}\tilde{F}_A\tom_A=dF=\sum_{A=1}^{n}F_A\omega_A=\sum_{A=1}^{n}\bar{\nabla}_{A}F\;\omega_A
\end{equation*}
gives the relation
\begin{equation}\label{eq27}
	\tilde{F}_A=\tn_{A}F=e^{-\rho}\bar{\nabla}_{A}F=e^{-\rho}F_A,~~\mbox{ for } A=1,\cdots, n.
\end{equation}
Further, we have
\begin{equation}\label{eq-con-lap}
	\Delta_{\tilde{g}_N}F=e^{-2\rho}\Bigl(\Delta_{g_N}F+(n-2)\bar{\nabla} \rho \cdot \bar{\nabla} F\Bigr).
\end{equation}

\subsection{Conformal relations: the submanifolds}
Let $M$ be an $m$-dimensional submanifold of $N$ immersed by $\phi$.

When considering the metric $g_N$, $M$ has an induced metric $g_M=\phi^{\ast}g_N$.
Choose $\{e_A\}_{A=1}^n$ such that $\{e_i\}_{i=1}^m$ are tangent\footnote{Here we identify $e_i$ with $\phi_{\ast}(e_i)$} to $M$
and $\{e_{\alpha}\}_{\alpha=m+1}^{n}$ are normal to $M$.
Denote $\phi^{\ast}\omega_{A}=\theta_{A}, \phi^{\ast}\omega_{AB}=\theta_{AB}$.
Then restricted to $(M,g_M)$, we have (cf. \cite{Chern1968})
\begin{equation}\label{eq23}
\theta_{\alpha}=0, \quad\theta_{i\alpha}=\sum_j h_{ij}^{\alpha}\theta_{j},
\end{equation}
and
\begin{equation}
\left\{\begin{aligned}\label{eq22}
&d\theta_{i}=\sum_{j}\theta_{ij}\wedge\theta_{j}, \quad\theta_{ij}+\theta_{ji}=0,\\
& d\theta_{ij}-\sum_k\theta_{ik}\wedge\theta_{kj}=-\frac{1}{2}\sum_{k,l}R_{ijkl}\,\theta_k\wedge\theta_l,
\end{aligned}
\right.
\end{equation}
where $R_{ijkl}$ are components of the curvature tensor of $(M, g_M)$ and $h_{ij}^{\alpha}$ are components of the second fundamental form of $(M, g_M)$ in $(N, g_N)$.

We denote the mean curvature vector of $M$ (precisely, of $\phi$) by
\begin{equation*}
	\mathbf{H}_{\phi}=\frac{1}{n}
\sum\limits_\alpha(\sum\limits_i h_{ii}^{\alpha})e_\alpha=\sum\limits_\alpha H^{\alpha}e_\alpha.
\end{equation*}
We sometimes omit the subscript $\phi$ and just write $\mathbf{H}$ when the immersion is clear.

Similarly, we consider the induced metric $\tg_M=\phi^{\ast}\tg_N$,
and denote $\phi^{\ast}\tom_{A}=\tt_{A},\phi^{\ast}\tom_{AB}=\tt_{AB}$. Then restricted to $(M,\tg_M)$, we have
\begin{equation}\label{eq29}\tt_{\alpha}=0, \quad\tt_{i\alpha}=\sum_j \th_{ij}^{\alpha}\tt_{j},
\end{equation}
and
\begin{equation}
\left\{
\begin{aligned}\label{eq28}
&\displaystyle d\tt_{i}=\sum_{j}\tt_{ij}\wedge\tt_{j}, \quad\tt_{ij}+\tt_{ji}=0, \\
\displaystyle &d\tt_{ij}-\sum_k\tt_{ik}\wedge\tt_{kj}=-\frac{1}{2}\sum_{k,l}\tR_{ijkl}\,\tt_k\wedge\tt_l,
\end{aligned}
\right.\end{equation}
where $\tR_{ijkl}$ are components of the curvature tensor of $(M, \tg_M)$ and $\th_{ij}^{\alpha}$ are components of the second fundamental form of $(M, \tg_M)$ in $(N, \tg_N)$.

By pulling back \eqref{eq26} to $M$ by $\phi$ and using \eqref{eq23} and \eqref{eq29}, we obtain the following relation:
\begin{equation}\label{eq208}
\th_{ij}^{\alpha}=e^{-\rho}(h_{ij}^{\alpha}-\rho_{\alpha}\delta_{ij}), \quad\tH^{\alpha}=e^{-\rho}(H^{\alpha}-\rho_{\alpha}).
\end{equation}
Combining \eqref{eq26}, \eqref{eq27}, \eqref{eq22},  and \eqref{eq28}, we derive the following relation:
\begin{equation}\label{eq210}
e^{2\rho}\tR=R-(m-2)(m-1)|\nabla \rho|^2-2(m-1)\Delta \rho,
\end{equation}
where $R$ (or $\tR$ resp.) is the scalar curvature with respect to $g_M$ (or $\tg_M$ resp.).

The following proposition gives the relation between the conformal factor and the mean curvature.
\begin{prop}[cf. \cite{CW19a}*{Proposition 3.3 and Remark 3.4}]\label{prop3.2}
Let $(N, g_N)$ be an $n$-dimensional Riemannian manifold (possibly not complete) which admits a conformal immersion $\gamma$ in the sphere $(\mathbb{S}^n, h_1)$.
Assume $\gamma^{\ast}h_1=e^{2\rho}g_N$.
Let $M$ be an $m$-dimensional submanifold immersed in $(N,g_N)$ by $\phi$. Then we have
\begin{equation}\label{eq-prop-1}
	e^{2\rho}=\Bigl(|\mathbf{H}_\phi|^2+\bar{\mathcal{R}}_\phi\Bigr)-\frac{2}{m}\Delta \rho-\frac{m-2}{m}|\nabla\rho|^2-|(\bar{\nabla}\rho)^\bot-\mathbf{H}_\phi|^2,
\end{equation}
where
\begin{equation}\label{eq-prop-2}
	\bar{\mathcal{R}}_\phi=\frac{1}{m(m-1)}\sum_{i,j}K_N(e_i,e_j)
\end{equation}
for a local orthonormal tangent frame $\{e_i\}$ on $M$, and $K_N$ is the sectional curvature of $(N, g_N)$.

In particular, when $(N,g_N)=(\R, h_c)$, we have
\begin{equation}\label{eq-prop-3}
	e^{2\rho}=(|\mathbf{H}|^2+c)-\frac{2}{m}\Delta \rho-\frac{m-2}{m}|\nabla\rho|^2-|(\bar{\nabla}\rho)^\bot-\mathbf{H}|^2.
\end{equation}
\end{prop}
\begin{proof}
	Consider two metric  $g_M=\phi^{\ast}g_N $ and $\tg_M=\phi^{\ast}(\gamma^{\ast}h_1)=(\gamma\circ \phi)^{\ast}h_1$ on $M$. The Gauss equations for the immersions $\phi$ and $\gamma\circ \phi$ imply
	\begin{align}
	R&=\sum_{i,j}K_N(e_i,e_j)+m^2|\mathbf{H}|^2-\sum_{i,j,\alpha}(h_{ij}^{\alpha})^2,\label{G1}\\
	\tR&=m(m-1)+m^2|\tilde{\mathbf{H}}|^2-\sum_{i,j,\alpha}(\th_{ij}^{\alpha})^2.\label{G2}
	\end{align}
	Hence, we derive \eqref{eq-prop-1} by using \eqref{eq208}, \eqref{eq210}, \eqref{G1} and \eqref{G2}.
\end{proof}

When $(N,g_N)=(\R, h_c)$, we usually denote the immersion of $M$ to $\R$ by $x$ instead of $\phi$, and $x$ can be viewed as the \emph{position vector}.
Then we have (cf. \cites{Tak66, Chern1968}):
\begin{equation*}
dx=\sum_{i}\theta_ie_i,\quad de_i=\sum_{j}\theta_{ij}e_j+\sum_{j}h_{ij}^{\alpha}\theta_{j}e_{\alpha}-c\theta_ix,\quad
de_{\alpha}=-\sum_{i,j}h_{ij}^{\alpha}\theta_{j}e_i+\sum_{\beta}\theta_{\alpha\beta}e_{\beta},
\end{equation*}
from which we obtain
\begin{equation}\label{eq_d2x}
x_i=e_i,\quad x_{ij}=\sum_{\alpha}h_{ij}^{\alpha}e_{\alpha}-c\delta_{ij}x,\quad \Delta x=m\mathbf{H}-mcx.
\end{equation}

\subsection{The generalization of the mean curvature.}\label{Sect:2.3}
Given a symmetric $(0,2)$-tensor $T=(T_{ij})$ on a (closed) manifold $M$, 
we define a normal vector field $H_T$ associated with $T$ by
\begin{equation*}
	H_T=\sum_{i=1}^mA(Te_i,e_i)=\sum_{\substack{m+1\leq \alpha\leq n\\ 1\leq i, j\leq m}}h_{ij}^{\alpha}T_{ij}e_{\alpha}.
\end{equation*}
We call $M$ is $T$-minimal in $N$ if $H_T\equiv 0$.

If we additionally assume that $T$ is  divergence-free, then we can define a self-adjoint operator $L_T$ associated to $T$ by
\begin{equation*}
	L_Tu=\dive(T\nabla u), u\in C^{\infty}(M),
\end{equation*}
where $T$ is regarded as a $(1, 1)$-tensor.

Now $H_T$ is a generalization of the mean curvature and the higher order mean curvatures, while $L_T$ is a generalization of the Laplacian.
Indeed, $H_T=m\mathbf{H}, L_T=\Delta$ when $T=I$ the identity.
Both $H_T$ and $L_T$ have been involved a lot of results (e.g., \cites{Gro04, CW19a}).

For a closed submanifold $M$ of $\mathbb{R}^n$, from \eqref{eq_d2x} one can easily obtain
\begin{equation}\label{eq-2.21}
	L_Tx=H_T, \quad \frac{1}{2}L_T|x|^2=\l x, H_T\r +\tr T
\end{equation}
and then the generalized Hsiung–Minkowski formula
\begin{equation}\label{eq-HM}
	0=\frac{1}{2}\int_M L_T|x|^2=\int_M\l x, H_T\r +\int_M \tr T
\end{equation}
by integrating the second equation of \eqref{eq-2.21}.

\section{Proofs for the Neumann eigenvalues}\label{Sect-3}
In this section, we prove Theorem \ref{thm-conf-h}. We divide the proof into two parts.
In Sect.~\ref{Sect3.1}, we prove the inequality in \eqref{eq-conf-h}; this is an easier part.
In Sect.~\ref{Sect3.2}, we discuss the equality case; this part needs more careful analysis.
We give a corollary to end the whole section.
\subsection{The inequality in \eqref{eq-conf-h}}\label{Sect3.1}
We have the variational characterization of $\la_i (i\ge 1)$:
\begin{equation}\label{eq-vc}
	\la_i=\inf_{u\in H^1(M)\setminus \{0\}}\left\{\frac{\int_M |\nabla u|^2}{\int_M u^2}\right|\left.\int_M uu_j=0, j=0, \cdots, i-1\right\},
\end{equation}
where $\{u_i\}_{i\ge 0}$ is an orthonormal set of eigenfunctions satisfying $\Delta u_i=-\la_i u_i$.

To construct suitable test functions, we recall the following lemma.
\begin{lem}[see \cites{LY82,ESI00}]\label{lem4.1}
	Let $\phi: (M^m,g)\to (\mathbb{S}^n, h_1)$ be a conformal map. Then there exists $\gamma\in \confs $ such that the immersion $\Phi=\gamma\circ \phi=(\Phi^1, \cdots, \Phi^{n+1})$ satisfies
	\begin{equation}\label{eq-4.1}
	\int_M \Phi^A=0,  \mbox{ for } A=1,\ldots,n+1.
	\end{equation}
	\end{lem}

Now we can start the proof. \eqref{eq-4.1} means that $\Phi^A$ is $L^{2}$-orthogonal to the first eigenfunction $u_0$ (a nonzero constant).
By using the QR-decomposition via the Gram–Schmidt process, we can further assume that $\{\Phi^1,\cdots, \Phi^{n+1}\}$ satisfies
\begin{equation}\label{eq-ortho-i}
	\int_M \Phi^A u_B=0,  \mbox{ for } 1\le B< A\le n+1.
\end{equation}
Indeed, if we denote
\begin{equation*}
	d_{AB}=\int_M \Phi^A u_B,  \mbox{ for } 1\le A, B\le n+1,
\end{equation*}
then the matrix $D=(d_{AB})=QR$, where $Q=(q_{AB})\in O(n+1)$ and $R$ is an upper triangular matrix.
This is equivalent to $R=Q^TD$. Hence, if we replace the old orthonormal basis $(E_1,\cdots, E_n)$ of $\mathbb{R}^{n+1}$ by the new one $(E_1',\cdots,E_{n+1}')=(E_1,\cdots, E_n)Q^T$, then these new coordinate functions, still denoted by $\Phi^A$, satisfy \eqref{eq-ortho-i}.

Hence, by the variational characterization \eqref{eq-vc},
for each $A=1,\cdots, n+1$, we have
	\begin{equation}\label{eq-4.5}
		\lambda_{A}\int_M (\Phi^A)^2\le \int_M|\nabla \Phi^A|^2.
	\end{equation}
	On the other hand, if we denote $\tilde{g}=\Phi^{\ast}h_1=fg$, then
	\begin{equation*}
		0\le |\tilde{\nabla} \Phi^A|^2\le 1-(\Phi^A)^2,\quad \sum_{A=1}^{n+1}|\tilde{\nabla} \Phi^A|^2=m,\quad \sum_{A=1}^{n+1}|\nabla\Phi^A|^2=mf,
	\end{equation*}
	where $\tilde{\nabla}$ is the gradient operator on $M$ with respect to $\tilde{g}$.
	Noting that $\sum_{A=1}^{n+1}\Phi_i^2=1$, from \eqref{eq-4.5} we obtain
\begin{align}
	V(M,g)=\int_Mdv_g\le{} &\sum_{A=1}^{m}\frac{1}{\lambda_{A}}\int_Mf|\tilde{\nabla} \Phi^A|^2+\sum_{A=m+1}^{n+1}\frac{1}{\lambda_{A}}\int_Mf|\tilde{\nabla} \Phi^A|^2\label{eq4.6}\\
	\le{} &\sum_{A=1}^{m}\frac{1}{\lambda_{A}}\int_Mf|\tilde{\nabla} \Phi^A|^2+\frac{1}{\lambda_{m}}\int_Mf\sum_{A=m+1}^{n+1}|\tilde{\nabla} \Phi^A|^2\label{eq4.7}\\
	={} &\sum_{A=1}^{m}\frac{1}{\lambda_{A}}\int_Mf|\tilde{\nabla} \Phi^A|^2+\frac{1}{\lambda_{m}}\int_Mf(m-\sum_{A=1}^{m}|\tilde{\nabla} \Phi^A|^2)\nonumber\\
	={} &\sum_{A=1}^{m}\frac{1}{\lambda_{A}}\int_Mf|\tilde{\nabla} \Phi^A|^2+\frac{1}{\lambda_{m}}\int_Mf\sum_{A=1}^{m}(1-|\tilde{\nabla} \Phi^A|^2)\nonumber\\
	\le{} &\sum_{A=1}^{m}\frac{1}{\lambda_{A}}\int_Mf|\tilde{\nabla} \Phi^A|^2+\sum_{A=1}^{m}\frac{1}{\lambda_{A}}\int_Mf(1-|\tilde{\nabla} \Phi^A|^2)\label{eq4.8}\\
	={} &\sum_{A=1}^{m}\frac{1}{\lambda_{A}}\int_M f.\label{eq-4.12}
\end{align}
By using the H\"{o}lder inequality, we have
\begin{align}
	\mathfrak{H}(\lambda_1,\cdots,\lambda_m)&\le \frac{m}{V(M,g)}\int_Mf\,dv_g\nonumber\\
	&\le \frac{m}{V(M,g)}\Bigl(\int_Mf^{m/2}\,dv_g\Bigr)^{2/m}\bigl(V(M)\bigr)^{\frac{m-2}{m}}\nonumber\\
	&= m\Bigl(\frac{V(M, \tilde{g})}{V(M,g)}\Bigr)^{2/m}\nonumber\\
	&\le m\Bigl(\frac{\sup_{\gamma\in \confs}V(M, (\gamma\circ\phi)^{\ast}h_1)}{V(M,g)}\Bigr)^{2/m}.\label{eq-4.13}
\end{align}
Taking the infimum over all non-degenerate conformal maps $\phi: M\to \mathbb{S}^n$, we obtain the desired inequality.

\subsection{Equality in \eqref{eq-conf-h}}\label{Sect3.2}
Now we discuss the equality.
First, we assume the equality holds.
Choose a sequence of conformal maps $\phi_j: M \to \mathbb{S}^n$ such that 
	\begin{equation*}
		\lim_{j\to \infty}V_{c}(M,n,\phi_j ) = V_c(M, n),
	\end{equation*}
	and by composing with a conformal transformation of the sphere we may assume
	\begin{equation}\label{eq-4.15}
			\int_M \phi_j^A u_B=0,  \mbox{ for } 0\le B< A\le n+1.
	\end{equation}

	We denote $mf_j=\sum_{A=1}^{n+1}|\nabla \phi_j^A|^2$ and write $\mathfrak{H}(\lambda_1,\cdots, \lambda_m)$ as $\mathfrak{H}$ for short.

	We reproduce the previous steps from \eqref{eq4.6} to \eqref{eq-4.13} for each $\phi_j$ and then take the limit.
	It follows that all the inequalities must be sharp.
	Indeed, noticing \eqref{eq4.6} and \eqref{eq-4.12}, we have
	\begin{align*}
		V(M,g)=\int_M\sum_{A=1}^{n+1}(\phi_j^A)^2&\le \mathfrak{H}^{-1}\int_M mf_j\\
			&\le \mathfrak{H}^{-1}m\Bigl(\int_Mf_j^{m/2}\,dv_g\Bigr)^{2/m}\bigl(V(M)\bigr)^{\frac{m-2}{m}}\\
			&\le \mathfrak{H}^{-1}m \bigl(V_c(M,n,\phi_j)\bigr)^{2/m}\bigl(V(M)\bigr)^{\frac{m-2}{m}}.
	\end{align*}
Letting $j\to \infty$ and using $ \mathfrak{H} =m\Bigl(\frac{V_c(M,n)}{V(M,g)}\Bigr)^{2/m}$, we obtain
\begin{align*}
	V(M,g)&=\lim_{j\to \infty}\int_M\sum_{A=1}^{n+1}(\phi_j^A)^2=\mathfrak{H}^{-1}\lim_{j\to \infty}\int_M \sum_{A=1}^{n+1}|\nabla \phi_j^A|^2\\
	&=\mathfrak{H}^{-1}\lim_{j\to \infty}\Bigl(\int_M\bigl(\sum_{A=1}^{n+1}|\nabla \phi_j^A|^2\bigr)^{m/2}\,dv_g\Bigr)^{2/m}\bigl(V(M)\bigr)^{\frac{m-2}{m}}\\
	&=V(M,g).
\end{align*}

Hence, for any fixed $A$, $\{\phi_j^A\}$ is a bounded sequence in $W^{1, m}(M)$.
Due to the compact inclusion $W^{1, m}(M)\hookrightarrow L^2(M)$, by passing to a subsequence we can assume that $\{\phi_j^A\}$ converges weakly in $W^{1, m}(M)$, strongly in $L^2(M)$, and pointwise a.e., to a map $\psi^A: M\to \mathbb{R}$.
We have
\begin{gather}
	\sum_{A=1}^{n+1}(\psi^A)^2=1, \mbox{ a.e. on $M$};\nonumber\\
	\int_{M} (\phi_j^A)^2
	\le\frac{1}{\lambda_A}\int_M |\nabla \phi_j^A|^2, \mbox{ for $A=1,\cdots, n$};\label{eq-4.22}\\
	\lim_{j\to\infty}\int_M (\phi_j^A)^2=\lim_{j\to\infty}\frac{1}{\lambda_A}\int_M |\nabla \phi_j^A|^2.\label{eq-4.23}
\end{gather}

\textbf{Claim 1.} For each $A$ fixed, $\{\phi_j^A\}$ converges to $\psi^A$ strongly in $W^{1, 2}(M)$; $\psi^A$ belongs to the eigenspace $E_{\lambda_A}$ corresponding to the eigenvalue $\lambda_A$, and then $\psi^A\in C^{\infty}(M)$.

\begin{proof}
Due to the weak convergence, \eqref{eq-4.15} implies
\begin{equation}
	\int_M \psi^A u_B=0,  \mbox{ for } 0\le B< A\le n+1.\label{eq-4.24}
\end{equation}
Then it follows from \eqref{eq-4.22} and \eqref{eq-4.23} that
\begin{equation*}
	\lim_{j\to\infty}\frac{1}{\lambda_A}\int_M |\nabla \phi_j^A|^2=\lim_{j\to\infty}\int_{M} (\phi_j^A)^2=\int_{M} (\psi^A)^2\le \frac{1}{\lambda_A}\int_M |\nabla \psi^A|^2.
\end{equation*}
On the other hand, $\phi_j^A \to \psi^A$ weakly in $W^{1,m}(M)$ implies the weak lower semicontinuity
	\begin{equation*}
		\int_{M}|\nabla \psi^A|^2\le\lim_{j\to\infty}\int_M |\nabla \phi_j^A|^2.
	\end{equation*}
Therefore, we must have
\begin{equation*}
	\lim_{j\to\infty}\int_M |\nabla \phi_j^A|^2=\int_{M}|\nabla \psi^A|^2,\quad \lambda_A\int_{M} (\psi^A)^2=\int_M |\nabla \psi^A|^2.
\end{equation*}
Noticing \eqref{eq-4.24}, these equalities prove Claim 1.
\end{proof}

Now we can continue the proof.
	Denote $mf=\sum_{A=1}^m|\nabla\psi^A|^2$.
	By taking the limit, all the inequalities
	from \eqref{eq4.6} to \eqref{eq-4.13} for $\phi$ become the equalities, that is,
	\begin{align}
		V(M,g)=\int_M\sum_{A=1}^m|\psi^A|^2={} &\sum_{A=1}^{m}\frac{1}{\lambda_{A}}\int_M |\nabla \psi^A|^2+\sum_{A=m+1}^{n+1}\frac{1}{\lambda_{A}}\int_M|\nabla \psi^A|^2\nonumber\\
		={} &\sum_{A=1}^{m}\frac{1}{\lambda_{A}}\int_M|\nabla \psi^A|^2+\frac{1}{\lambda_{m}}\int_M\sum_{A=m+1}^{n+1}|\nabla \psi^A|^2\label{eq4.7-psi}\\
		={} &\sum_{A=1}^{m}\frac{1}{\lambda_{A}}\int_M|\nabla \psi^A|^2+\frac{1}{\lambda_{m}}\int_M  f\sum_{A=1}^{m}\bigl(1-\frac{1}{f}|\nabla \psi^A|^2\bigr)\nonumber\\
		={}&\sum_{A=1}^{m}\frac{1}{\lambda_{A}}\int_M|\nabla \psi^A|^2+\sum_{A=1}^{m}\frac{1}{\lambda_{A}}\int_Mf\bigl(1-\frac{1}{f}|\nabla \psi^A|^2\bigr)\label{eq4.8-psi}\\
		={} &\sum_{A=1}^{m}\frac{1}{\lambda_{A}}\int_M f= \mathfrak{H}^{-1}m\Bigl(\int_Mf^{m/2}\,dv_g\Bigr)^{2/m}\bigl(V(M)\bigr)^{\frac{m-2}{m}}\nonumber\\
		={}&\mathfrak{H}^{-1}m\Bigl(V_c(M,n)\Bigr)^{2/m}\bigl(V(M)\bigr)^{\frac{m-2}{m}}= V(M, g).\nonumber
	\end{align}

	The strong convergence in $W^{1,2}(M)$ proves that
	\begin{equation*}
		\psi^{\ast}h_1=\lim_{j\to\infty}\phi_j^{\ast}h_1=\lim_{j\to\infty}f_jg=fg,
	\end{equation*}
which means that  $\psi: M\to \mathbb{S}^n$ is a conformal  map with the conformal factor $f$.

Next, we prove the following lemma.
\begin{lem}\label{lem-3.4}
	If the equality holds in \eqref{eq-conf-h}, then $\lambda_1=\cdots=\lambda_m$ and $f=\lambda_1/m$.
\end{lem}
\begin{proof}
	Recalling Claim 1, we have
	\begin{equation}\label{eq-3.12}
		\Delta \psi^A=-\la_A\psi^A,
	\end{equation}
	and then
	\begin{equation}\label{eq3.13}
		mf=\sum_{A=1}^{n+1}|\nabla \psi^A|^2=\sum_{A=1}^{n+1}\Bigl( \frac{1}{2}\Delta (\psi^A)^2-\psi^A\Delta \psi^A\Bigr)
			=\sum_{A=1}^{n+1}\lambda_A(\psi^A)^2.
	\end{equation}

	We point out that $|\nabla \psi^A|^2/f\equiv 1$ is impossible. Otherwise, $1\equiv |\nabla \psi^A|^2/f=|\tilde{\nabla} \psi^A|^2\le 1-(\psi^A)^2$, which means $\psi^A\equiv 0$ and then $|\nabla \psi^A|^2/f=0$, a contradiction.
	Hence, equality in \eqref{eq4.8-psi} gives $\lambda_t=\lambda_m$ for all $t\le m-1$.

	If $|\nabla \psi^A|^2\equiv 0$, then $\psi^{A}$ is constant on $M$, which implies $\psi^{A}\equiv 0$ by \eqref{eq-3.12}.
	If $\psi^s\not\equiv 0$ for some $s\ge m+1$, then equality in \eqref{eq4.7-psi} gives $\lambda_m=\lambda_s$.
	Hence, we conclude from \eqref{eq3.13} that
	\begin{equation*}
		mf=\lambda_m\sum_{\stackrel{1\le A\le n+1}{\psi^A\not\equiv 0}}(\psi^A)^2=\lambda_m=\lambda_1.
	\end{equation*}
	The proof is complete.
\end{proof}

By scaling we may assume $\lambda_1=m$. Then $\psi: M\to \mathbb{S}^n$ is an isometric immersion. Since each $\psi^A\in E_{\lambda_1}$, we know that $\psi$ is minimal by the well-known Takahashi's theorem (\cite{Tak66}*{Theorem 3}).

Conversely, the sufficiency for the equality is directly from the following theorem  when $M$ is closed.
\begin{thm}[\cite{ESI86}*{Theorem 1.1}]\label{thm-ESI-1.1}
	Let $(M, g)$ be an $m$-dimensional closed Riemannian manifold.
	Suppose there exists a minimal isometric immersion $\phi$ of $(M, g)$ in $\mathbb{S}^n$. Then
	\begin{equation*}
		V(M, g) = V_c(M, n, \phi)\ge V_c(M,n).
	\end{equation*}
	Moreover, if $(M,g)$ is not isometric to $(\mathbb{S}^m, h_1)$, then $V(M,g)> V(M, (\gamma\circ\phi)^\ast h_1)$ for all $\gamma\in G\setminus O(n+1)$.
	Here $G=\{\gamma_a|a\in B^{n+1}\}$ is a subgroup of $\confs$, see the remark below for details.
\end{thm}

\begin{rem}\label{rem-4.1}
	For each $a\in B^{n+1}$, we can define a conformal map $\gamma_a$ on $\mathbb{S}^n(1)$ (cf. \cite{MR86}):
		\begin{equation*}
		\gamma_a(x)=\frac{x+(\mu f+\lambda)a}{\lambda(1+f)}, \quad \forall x\in \mathbb{S}^n(1),
		\end{equation*}
		where $B^{n+1}$ is the open unit ball in $\mathbb{R}^{n+1}$, $x$ is the position vector of $\mathbb{S}^n(1)$, and
		\begin{equation*}
		\lambda=(1-|a|^2)^{-1/2},\quad\mu=(\lambda-1)|a|^{-2},\quad f(x)=\l x, a\r.
		\end{equation*}
		When $a=0$, we set $\lambda=1,\mu=0, \gamma_0(x)=x.$

The proof of Theorem \ref{thm-ESI-1.1} used the following fact, which allows us to replace $\confs$ by the subgroup  $G=\{\gamma_a|a\in B^{n+1}\}$ when taking the supremum in the definition of $V(M,n,\phi)$.
(see \cite{ESI86}*{pp.~259}):

	For any $\gamma\in \confs$, there exists $a\in B^{n+1}$ and $r\in O(n+1)$ such that $\gamma=r\circ \gamma_a$.
\end{rem}

\subsection{A corollary for submanifolds of a sphere}
\begin{cor}\label{cor-4.1}
	(1) Let $(M, g)$ be an $m$-dimensional closed Riemannian manifold
	which can be minimally immersed into $\mathbb{S}^n$ by  $\phi$.
	If $(M,g)$ is not isometric to $(\mathbb{S}^m, h_1)$,
	then for any metric $\tilde{g}\in [g]$, we have
	\begin{equation}\label{eq-4.41}
		\mathfrak{H}(\tilde{g})V(M,\tilde{g})^{2/m}\le mV(M,g)^{2/m},
	\end{equation}
	where $[g]$ is the conformal class of $g$, and $\mathfrak{H}(\tilde{g})$ is the harmonic mean of the first $m$ nonzero eigenvalues with respect to $\tilde{g}$.

	Moreover, equality holds if and only if the minimal immersion $\phi$ is given by a subspace of the first eigenspace and $\tilde{g}=kg$ for some constant $k>0$.

	(2)  For any metric $g$ on $\mathbb{S}^m, g\in [h_1]$, we have
	\begin{equation*}
		\mathfrak{H}(g)V(\mathbb{S}^m,g)^{2/m}\le \mathfrak{H}(h_1)V(\mathbb{S}^m,h_1)^{2/m}.
	\end{equation*}
	Moreover, equality holds if and only if  $g=kh_1$ for some constant $k>0$.
\end{cor}

\begin{proof}
	(1) The inequality \eqref{eq-4.41} is from Theorem \ref{thm-conf-h} and Theorem \ref{thm-ESI-1.1}.
	Now assume that the equality holds.
	After rescaling, we can assume that $\mathfrak{H}(\tilde{g})=m$. Then $V(M,\tilde{g})=V(M,g)$ and there is a minimal immersion
	$\psi: (M, \tilde{g})\to \mathbb{S}^n$ with $\lambda_1(\tilde{g})=\cdots=\lambda_m(\tilde{g})=m$
	by the case of equality in Theorem \ref{thm-conf-h}.

	By reviewing the proof of Theorem \ref{thm-conf-h}, we can assume $\psi=\gamma_a\circ \phi$ for some $a\in B^{n+1}$. 
	Since $(M, g)$ is not isometric to $(\mathbb{S}^m,h_1)$,
	$\gamma_a$ must be the identity map by Theorem \ref{thm-ESI-1.1} (and Remark \ref{rem-4.1}).
	This follows that $\tilde{g}=\psi^\ast h_1=\phi^\ast h_1=g$.

	(2) This assertion follows from Theorem \ref{thm-conf-h}, just noting that $V_c(\mathbb{S}^m, g)=V_c(\mathbb{S}^m, h_1)=V(\mathbb{S}^m, h_1)$ and $\mathfrak{H}(h_1)=\lambda_1(h_1)=m$.
\end{proof}

\section{Proofs for the Steklov eigenvalues}\label{Sect-4}
In this section, we prove Theorems \ref{thm-FS-h}, \ref{thm-FS-h-2} and \ref{thm-FS-h-3}.
Since the main approaches are the same as in Sect.~\ref{Sect-3}, we will omit common details but emphasize the differences.

\subsection{The inequality in \eqref{eq-FS-h}} \label{Sect4.1}
We have the variational characterization of the Steklov eigenvalue $\sigma_i (i\ge 1)$ (cf. \cites{Xio21, Xio22})
\begin{equation*}
	\sigma_i=\inf_{\stackrel{u\in H^1(M)\setminus \{0\}}{u|_{\partial M}\neq 0}}\left\{\frac{\int_M|\nabla u|^2}{\int_{\partial M} u^2}\right|\left.\int_{\partial M} uv_j=0, j=0, \cdots, i-1\right\},
\end{equation*}
where $\{v_i\}_{i\ge 0}$ is an orthonormal set of Steklov eigenfunctions.

Let $\phi: (M, g)\to B^n$  be a conformal map with $\phi(\partial M)\subset \partial B^n$. By similar arguments as in Sect.~\ref{Sect3.1}, there exists $\gamma\in \confb$ such that
$\Phi=\gamma\circ \phi=(\Phi_1,\cdots,\Phi_n)$ satisfies that
\begin{gather*}
\int_{\partial M}\Phi^A=0,\mbox{ for } A=1,\ldots, n.\\
	\int_{\partial M} \Phi^A v_B=0,  \mbox{ for } 1\le B< A\le n.
\end{gather*}
Let $\hat{\Phi}^A$ be a harmonic extension of $\Phi^A|_{\partial M}$.
Now for each $A=1,\cdots, n$, we have
	\begin{equation*}
		\sigma_{A}\int_{\partial M} (\Phi^A)^2
		\le \int_M|\nabla \hat{\Phi}^A|^2
		\le \int_M|\nabla \Phi^A|^2.
	\end{equation*}
	On the other hand, if we denote $\tilde{g}=\Phi^{\ast}g_{B^n}=fg$, then
	\begin{equation*}
		0\le |\tilde{\nabla} \Phi^A|^2\le 1,\quad \sum_{A=1}^{n}|\tilde{\nabla} \Phi^A|^2=m,\quad \sum_{A=1}^{n}|\nabla\Phi^A|^2=mf,
	\end{equation*}
	where $\tilde{\nabla}$ is the gradient operator on $M$ with respect to $\tilde{g}$.

	An analogous computation as in Sect.~\ref{Sect3.1} shows that
\begin{align*}
	V(\partial M)=\int_{\partial M}1\le{} &\sum_{A=1}^{m-1}\frac{1}{\sigma_{A}}\int_Mf|\tilde{\nabla} \Phi^A|^2+\sum_{A=m}^{n}\frac{1}{\sigma_{A}}\int_Mf|\tilde{\nabla} \Phi^A|^2\\
	\l{} &\sum_{A=1}^{m-1}\frac{1}{\sigma_{A}}\int_Mf|\tilde{\nabla} \Phi^A|^2+\frac{1}{\sigma_{m}}\int_Mf\sum_{A=m}^{n}|\tilde{\nabla} \Phi^A|^2\\
	={} &\sum_{A=1}^{m-1}\frac{1}{\sigma_{A}}\int_Mf|\tilde{\nabla} \Phi^A|^2+\frac{1}{\sigma_{m}}\int_Mf(m-\sum_{A=1}^{m-1}|\tilde{\nabla} \Phi^A|^2)\\
	={} &\sum_{A=1}^{m-1}\frac{1}{\sigma_{A}}\int_Mf|\tilde{\nabla} \Phi^A|^2+\frac{1}{\sigma_{m}}\int_Mf\sum_{A=1}^{m-1}(1-|\tilde{\nabla} \Phi^A|^2)+\frac{1}{\sigma_{m}}\int_Mf\\
	\le{} &\sum_{A=1}^{m-1}\frac{1}{\sigma_{A}}\int_Mf|\tilde{\nabla} \Phi^A|^2+\sum_{A=1}^{m-1}\frac{1}{\sigma_{A}}\int_Mf(1-|\tilde{\nabla} \Phi^A|^2)+\frac{1}{\sigma_{m}}\int_Mf\\
	={} &\sum_{A=1}^{m}\frac{1}{\sigma_{A}}\int_M f.
\end{align*}
Hence, by using the H\"{o}lder inequality, we have
\begin{align}
	\mathfrak{H}(\sigma_1,\cdots,\sigma_m) V(\partial M) &\le m\int_Mf\,dv_g\label{eq-5.10}\\
	&\le m\Bigl(\int_Mf^{m/2}\,dv_g\Bigr)^{2/m}\bigl(V(M)\bigr)^{\frac{m-2}{m}}\nonumber\\
	&= m\Bigl(V(M, \tilde{g})\Bigr)^{2/m}\bigl(V(M)\bigr)^{\frac{m-2}{m}}\nonumber\\
	&\le m\Bigl(\sup_{\gamma\in \confb}V(M, (\gamma\circ\phi)^{\ast}g_{B^n})\Bigr)^{2/m}\bigl(V(M)\bigr)^{\frac{m-2}{m}}.\nonumber
\end{align}
Taking the infimum over all non-degenerate conformal maps $\phi: M\to B^n$ with $\phi(\partial M)\subset \partial B^n$, we obtain the desired inequality.

\subsection{Equality in \eqref{eq-FS-h}}\label{Sect5.2}
Now, we assume the equality holds.
Choose a sequence of conformal maps $\phi_j: M \to B^n$ with $\phi_j(\partial M)\subset \partial B^n$ such that
	\begin{equation*}
		\lim_{j\to \infty}V_{rc}(M,n,\phi_j ) = V_{rc}(M, n),
	\end{equation*}
	and by composing with a conformal transformation of the sphere we may assume
	\begin{equation*}
		\int_{\partial M} \phi_j^A v_B=0,  \mbox{ for } 0\le B< A\le n.
	\end{equation*}
	Denote $mf_j=\sum_{A=1}^{n}|\nabla \phi_j^A|^2$ and write $\mathfrak{H}(\sigma_1,\cdots, \sigma_m)$ as $\mathfrak{H}$ for short.

	Like Sect.~\ref{Sect3.2}, we reproduce the steps in Sect.~\ref{Sect4.1} for each $\phi_j$ and then take the limit. It follows that all the inequalities must be sharp. In summary,
\begin{align*}
	V(\partial M)&=\lim_{j\to \infty}\int_M\sum_{A=1}^{n}(\phi_j^A)^2=\mathfrak{H}^{-1}\lim_{j\to \infty}\int_M \sum_{A=1}^{n}|\nabla \phi_j^A|^2\\
	&=\mathfrak{H}^{-1}\lim_{j\to \infty}\Bigl(\int_M\bigl(\sum_{A=1}^{n}|\nabla \phi_j^A|^2\bigr)^{m/2}\,dv_g\Bigr)^{2/m}\bigl(V(M)\bigr)^{\frac{m-2}{m}}\\
	&=V(\partial M).
\end{align*}

Hence, by passing to a subsequence we can assume that $\{\phi_j^A\}$ converges weakly in $W^{1, m}(M)$, strongly in $L^2(M)$, and pointwise a.e., to a map $\psi^A: M\to \mathbb{R}$.
We have
\begin{gather*}
	\sum_{A=1}^{n}(\psi^A)^2=1, \mbox{ a.e. on $\partial M$};\quad
	\sum_{A=1}^{n}(\psi^A)^2\le 1, \mbox{ a.e. on $M$};\\
	\int_{\partial M} (\phi_j^A)^2
	\le\frac{1}{\sigma_A}\int_M |\nabla \phi_j^A|^2, \mbox{ for $A=1,\cdots, n$};\\
	\lim_{j\to\infty}\sum_{A=1}^{n}\int_{\partial M} (\phi_j^A)^2=\lim_{j\to\infty}\sum_{A=1}^{n}\frac{1}{\sigma_A}\int_M |\nabla \phi_j^A|^2.
\end{gather*}

Similar to Claim 1 in Sect.~\ref{Sect3.2}, we can prove

\textbf{Claim 2.} For each $A$ fixed, $\{\phi_j^A\}$ converges to $\psi^A$ strongly in $W^{1, 2}(M)$; $\psi^A$ belongs to the eigenspace $E_{\sigma_A}$ corresponding to the eigenvalue $\sigma_A$, and then $\psi^A$ is harmonic.

Now we continue the proof. Denote $mf=\sum_{A=1}^n|\nabla\psi^A|^2$.
By taking the limit, we have
	\begin{align}
		V(\partial M)=\int_{\partial M}\sum_{A=1}^n|\psi^A|^2={} &\sum_{A=1}^{m-1}\frac{1}{\sigma_{A}}\int_M |\nabla \psi^A|^2+\sum_{A=m}^{n}\frac{1}{\sigma_{A}}\int_M|\nabla \psi^A|^2\nonumber\\
		={} &\sum_{A=1}^{m-1}\frac{1}{\sigma_{A}}\int_M|\nabla \psi^A|^2+\frac{1}{\sigma_{m}}\int_M\sum_{A=m}^{n}|\nabla \psi^A|^2\label{eq5.7-psi}\\
		={} &\sum_{A=1}^{m-1}\frac{1}{\sigma_{A}}\int_M|\nabla \psi^A|^2+\frac{1}{\sigma_{m}}\int_M  f\sum_{A=1}^{m-1}\bigl(1-\frac{1}{f}|\nabla \psi^A|^2\bigr)+\frac{1}{\sigma_{m}}\int_Mf\nonumber\\
		={}&\sum_{A=1}^{m-1}\frac{1}{\sigma_{A}}\int_M|\nabla \psi^A|^2+\sum_{A=1}^{m-1}\frac{1}{\sigma_{A}}\int_Mf\bigl(1-\frac{1}{f}|\nabla \psi^A|^2\bigr)+\frac{1}{\sigma_{m}}\int_Mf\label{eq5.8-psi}\\
		={} &\sum_{A=1}^{m}\frac{1}{\sigma_{A}}\int_M f\nonumber\\
		={}& \mathfrak{H}^{-1}m\Bigl(\int_Mf^{m/2}\,dv_g\Bigr)^{2/m}\bigl(V(M)\bigr)^{\frac{m-2}{m}}\label{eq-holder-psi}\\
		={}&\mathfrak{H}^{-1}m\Bigl(V_{rc}(M,n)\Bigr)^{2/m}\bigl(V(M)\bigr)^{\frac{m-2}{m}}= V(\partial M).\nonumber
	\end{align}

	The strong convergence in $W^{1,2}(M)$ proves that
	\begin{equation*}
		\tilde{g}:=\psi^{\ast}g_{B^n}=\lim_{j\to\infty}\phi_j^{\ast}g_{B^n}=\lim_{j\to\infty}f_jg=fg,
	\end{equation*}
which means that  $\psi: M\to B^n$ is a conformal map with the conformal factor $f$.
Moreover, $\psi$ is harmonic with $\psi(\partial M)\subset \partial B^n$.

When $m>2$, equality in \eqref{eq-holder-psi} implies $f$ is constant.
Hence, $\psi$ is an isometric immersion after rescaling ($f\equiv 1$ and $\sum_{A=1}^m\frac{1}{\sigma_A}=\frac{V(\partial M)}{V(M)}$), and then it is a minimal immersion since each coordinate function is harmonic.

The next lemma allows us to discuss the equality case further.
\begin{lem}\label{lem-5}
	The equality in \eqref{eq-FS-h} implies $\sigma_1=\dots=\sigma_m$.
\end{lem}
\begin{proof}
	If $\nabla \psi^s\equiv 0$ on $M$ for some $s (1\le s\le n)$, then $\psi(M)\subset B^{n-1}$ and $\psi(\partial M)\subset \partial B^{n-1}$.
	So without loss of generality we can assume that $\nabla \psi^s\not\equiv 0$ on $M$ for each $s$.
	Then $\psi^s\not\equiv 0$ on $\partial M$ for each $s$ since $\sigma_s \int_{\partial M}|\psi^s|^2=\int_M |\nabla \psi^s|^2$.
	Moreover, we immediately obtain $\sigma_s=\sigma_m$ for $s\ge m+1$ from \eqref{eq5.7-psi}.
	We denote the Levi-Civita connection with respect to the induced metric $\tilde{g}=\psi^{\ast}h_0$ by $\tilde{\nabla}$, then $\psi: (M, \tilde{g})\to \mathbb{R}^n$ is an isometric immersion.

	Since $\nu=(\tilde{\nabla}_{\nu} \psi^1, \dots, \tilde{\nabla}_{\nu} \psi^n)=\frac{1}{\sqrt{f}}(\nabla_{\nu} \psi^1, \dots, \nabla_{\nu} \psi^n)=\frac{1}{\sqrt{f}}(\sigma_1\psi^1,\dots,\sigma_n\psi^n)$ and $\nu$ is a unit vector, we have
	\begin{equation}\label{eq-lem4.1-1}
		\sum_{s=1}^{n}\sigma_s^2(\psi^s)^2=f  \mbox{ on $\partial M$}.
	\end{equation}

	Noting \eqref{eq5.8-psi}, if $|\nabla \psi^1|^2\not \equiv f$ on $M$, then $\sigma_1=\sigma_m$.

	If $|\nabla \psi^1|^2 \equiv f$ on $M$,
	then when restricted to $\partial M$,
	we have
	\begin{align}
		f=|\nabla \psi^1|^2&=|\nabla^{\partial M} \psi^1|^2+|\nabla_{\nu} \psi^1|^2\nonumber\\
		&=f|\tilde{\nabla}^{\partial M} \psi^1|^2+\sigma_1^2(\psi^1)^2\nonumber\\
		&\le f(1-(\psi^1)^2)+\sigma_1^2(\psi^1)^2,\label{eq-lem4.1-2}
	\end{align}
	which implies
	\begin{equation}\label{eq-lem4.1-3}
	\sigma_1^2\ge f  \mbox{ on $\partial M$},
	\end{equation}
	where \eqref{eq-lem4.1-2} holds since $\psi: (\partial M, \tilde{g})\to \partial B^n=S^{n-1}$ is an isometric immersion.
	On the other hand, we have
	\begin{equation}\label{eq-lem4.1-4}
	\sum_{s=1}^{n}(\psi^s)^2=1  \mbox{ on $\partial M$}.
	\end{equation}
	Now from \eqref{eq-lem4.1-1}, \eqref{eq-lem4.1-3} and \eqref{eq-lem4.1-4} we obtain
	$\sigma_1=\cdots=\sigma_m$ since we have assumed $\psi^s\not\equiv 0$ on $\partial M$.
	Moreover, we see that $f=\sigma_1$ is constant on $\partial M$.
\end{proof}

By scaling the metric we assume that $\sigma_1=1$, then due to Lemma \ref{lem-5} we have $\frac{\partial \psi}{\partial \nu}=\psi$ on $\partial M$, which follows that $\psi(M)$ meets $\partial B^n$ orthogonally along $\psi(\partial M)$ and $\psi$ is an isometry on $\partial M$.

\subsection{Proofs of Theorems \ref{thm-FS-h-2} and \ref{thm-FS-h-3}}
As pointed out in the proof of \cite{FS11}*{Theorem 2.3}, for any compact surface $M$ with boundary,
there exists a conformal branched cover  $\phi: M\to D$ with $\deg(\phi)\le \genus(M)+k$ (cf. \cites{Ahl50, Gab06}).

Since $\phi$ is proper we have $\phi(\partial M)\subset \partial D$.
For the same reason as in Sect.~\ref{Sect4.1},
the coordinate functions $\phi^1,\phi^2$ can be used as the test functions.
Hence, we have (cf. \eqref{eq-5.10})
\begin{equation}\label{eq-4.38}
	\mathfrak{H}(\sigma_1,\sigma_2) V(\partial M) \le \int_M\sum_{i=1}^2|\nabla \phi^i|^2\,dv_g
\end{equation}
Since $\phi$ is conformal and $\dim M=2$, the right side of \eqref{eq-4.38} equals to $2V(\phi(M))$
and bounded by $2(\genus(M)+k)\pi$ from above. This proves Theorem \ref{thm-FS-h-2}.

To prove Theorem \ref{thm-FS-h-3}, we just notice that the stereographic projection from $\mathbb{R}^2$ to $(\mathbb{S}^2,h_1)$ is conformal and then consider the conformal map $\psi: M \to (\mathbb{S}^2,h_1)$.
By using the coordinate functions $\psi^1,\psi^2,\psi^3$ (in $\mathbb{R}^3$) as the test functions, we have
\begin{align*}
	\mathfrak{H}(\sigma_1,\sigma_2) V(\partial M) &\le \int_M\sum_{i=1}^3|\nabla \psi^i|^2\,dv_g=2V(\psi(M))\\
	&<2 V(\mathbb{S}^2,h_1)\deg(\psi)\le 8\pi (\genus(M)+1).
\end{align*}
We complete the proof.

\section{Proofs of Reilly inequalities in space forms}\label{Sect-5}
In this section, we prove Theorem \ref{thm1.1} and its general version.

\subsection{The inequality in \eqref{eq_thm1}}\label{Sect-5.1}
For the immersion $x: M\to\R$, by analogous arguments as in Sect.~\ref{Sect3.1},
there exists a regular conformal map $\Gamma: \R\to \S\subset\mathbb{R}^{n+1}$ such that  the immersion $\Phi=\Gamma\circ x=(\Phi^1, \cdots, \Phi^{n+1})$ satisfies that
	\begin{equation*}
		\int_M \Phi^A u_B=0,  \mbox{ for } 0\le B< A\le n+1.
	\end{equation*}
Set $\Gamma^{\ast}h_1=e^{2\rho}h_c$, then $g_M=x^{\ast}h_c$, $\tg_M=(\Gamma\circ x)^{\ast}h_1.$

	By the variational characterization \eqref{eq-vc},
for each $A=1,\cdots, n+1$, we have
	\begin{equation*}
		\lambda_{A}\int_M (\Phi^A)^2\le \int_M|\nabla \Phi^A|^2
		=\int_M e^{2\rho}|\tilde{\nabla}\Phi^A|^2,
	\end{equation*}
	An analogous computation as in Sect.~\ref{Sect3.1} shows that
	\begin{align}
		V(M)=\int_M1\le{} &\sum_{A=1}^{m}\frac{1}{\lambda_{A}}\int_Me^{2\rho}|\tilde{\nabla}\Phi^A |^2+\sum_{A=m+1}^{n+1}\frac{1}{\lambda_{A}}\int_Me^{2\rho}|\tilde{\nabla}\Phi^A |^2\nonumber\\
		\le{} &\sum_{A=1}^{m}\frac{1}{\lambda_{A}}\int_Me^{2\rho}|\tilde{\nabla}\Phi^A |^2+\frac{1}{\lambda_{m}}\int_Me^{2\rho}\sum_{A=m+1}^{n+1}|\tilde{\nabla}\Phi^A |^2\nonumber\\
		={} &\sum_{A=1}^{m}\frac{1}{\lambda_{A}}\int_Me^{2\rho}|\tilde{\nabla}\Phi^A |^2+\frac{1}{\lambda_{m}}\int_Me^{2\rho}\sum_{A=1}^{m}(1-|\tilde{\nabla}\Phi^A |^2)\nonumber\\
		\le{} &\sum_{A=1}^{m}\frac{1}{\lambda_{A}}\int_Me^{2\rho}|\tilde{\nabla}\Phi^A |^2+\sum_{A=1}^{m}\frac{1}{\lambda_{A}}\int_Me^{2\rho}(1-|\tilde{\nabla}\Phi^A |^2)\label{eq3.8}\\
		={} &\sum_{A=1}^{m}\frac{1}{\lambda_{A}}\int_M e^{2\rho}.\nonumber
	\end{align}

	Hence, by using \eqref{eq-prop-3} we derive that
	\begin{equation*}
		\mathfrak{H}(\lambda_1,\cdots,\lambda_m)\le \frac{m}{V(M)}\int_M e^{2\rho}
		\le \frac{m}{V(M)}\int_M \Bigl(c+|\mathbf{H}|^2\Bigr).
	\end{equation*}

\subsection{Equality in \eqref{eq_thm1}}\label{Sect-5.3}
Firstly, we discuss the necessary conditions for equality in \eqref{eq_thm1}.

When $M$ is minimally immersed in $\mathbb{S}^n$, we know that $\lambda=m$ is an eigenvalue and its multiplicity is not less than $m$ since $\Delta x=-mx$.
Hence, $\mathfrak{H}(\lambda_1,\cdots,\lambda_m)=\frac{m}{V(M)}\int_M(c+|\mathbf{H}|^2)=m$ implies $\lambda_1=\dots=\lambda_m=m$, i.e., the immersion $x$ is given by a subspace of the first eigenspace.

When $c\neq 1$ or $M$ is not minimal in $\mathbb{S}^n$, following the proof of Lemma \ref{lem-3.4}, equality in \eqref{eq3.8} implies $\lambda_1=\dots=\lambda_m$.
This makes us reduce the equality case in \eqref{eq_thm1} to the equality case in \eqref{eq-Rei} and complete the proof for the necessity.

Next, we discuss the sufficient conditions for equality in \eqref{eq_thm1}. 
We already have
\begin{equation*}
	\lambda_1\le \mathfrak{H}(\lambda_1,\cdots,\lambda_m)\le \frac{m}{V(M)}\int_M(c+|\mathbf{H}|^2),
\end{equation*}

Noting Theorem \ref{thm-Rei}, from either case (1) or case (2) we can derive
\begin{equation*}
	\lambda_1= \frac{m}{V(M)}\int_M(c+|\mathbf{H}|^2),
\end{equation*}
so equality in \eqref{eq_thm1} is attained. This proves the sufficiency.

\subsection{A general version}
When the ambient is not the space form, we prove the following theorem, which can recover \cite{ESI92}*{Theorem 2}.

\begin{thm}\label{thm-3.6}
	Let $(N, g_N)$ be an $n$-dimensional Riemannian manifold (possibly not complete) which admits a conformal immersion in the sphere $(\mathbb{S}^n, h_1)$. Then, for any $m$-dimensional closed submanifold $M$ immersed in $(N, g_N)$ by $\phi$, we have
	\begin{equation*}
		\mathfrak{H}(\lambda_1,\cdots,\lambda_m)\le \frac{m}{V(M)}\int_M(|\mathbf{H}_{\phi}|^2+\bar{\mathcal{R}}_{\phi}),
	\end{equation*}
	where $\bar{\mathcal{R}}$ is defined by \eqref{eq-prop-2}.

	Moreover, equality holds if and only if $\lambda_1=\cdots=\lambda_m$ and $|\mathbf{H}_{\phi}|^2+\bar{\mathcal{R}}_{\phi}$ equals the constant $\lambda_1/m$,
\end{thm}

\begin{proof}
	The inequality follows from the same arguments as in Sect.~\ref{Sect-5.1} (just using \eqref{eq-prop-1} instead of \eqref{eq-prop-3} in the last step). 
	The sufficiency for the equality is obvious.
	The necessary for the equality is due to Lemma \ref{lem-3.4}. We omit the details.
\end{proof}

\section{Proofs of extrinsic estimates in the Euclidean space}\label{Sect6}
In this section, we prove Reilly-type inequalities for submanifolds of $\mathbb{R}^n$.
\begin{proof}[Proof of Theorem \ref{thm-ext-c}]
When $M$ is a submanifold of $\mathbb{R}^n$, without lost of generality, we assume that the coordinate functions satisfy
\begin{equation*}
	\int_{M}x^A=0 \mbox{ for $A=1,\dots, n$}.
\end{equation*}
We can assume further
\begin{equation*}
	\int_{M}x^Au_B= 0 \mbox{ for $1\le B<A \le n$}
\end{equation*}
by using the QR-decomposition.
We claim
\begin{equation}\label{eq-5.34}
	\mathfrak{H}(\lambda_1,\cdots,\lambda_m) \int_{M} |x|^2 \le m V(M).
\end{equation}
Indeed, \eqref{eq-5.34} can be derived from \eqref{eq3.8}, just noting that the conformal factor $e^{2\rho}\equiv 1$ since we don't need the conformal transformation,
and replacing the left side $V(M)=\int_{M}1$ by $\int_{M} |x|^2$ since $M$ may be not in a sphere.

Multiplying the both sides of \eqref{eq-5.34} by $\int_M|H_T|^2$, we have
\begin{align}
	m V(M)	\int_M|H_T|^2& \ge \mathfrak{H}(\lambda_1,\cdots,\lambda_m) \int_{M} |x|^2\int_M|H_T|^2\nonumber\\
	&\ge \mathfrak{H}(\lambda_1,\cdots,\lambda_m)\Bigl(\int_M|x|\cdot |H_T| \Bigr)^2\label{eq-6.2}\\
	&\ge \mathfrak{H}(\lambda_1,\cdots,\lambda_m)\Bigl(\int_M\l x, H_T\r \Bigr)^2\label{eq-6.3}\\
	&=  \mathfrak{H}(\lambda_1,\cdots,\lambda_m)\Bigl(\int_M\tr T \Bigr)^2,\nonumber
\end{align}
where we used the H\"{o}lder inequality, the Cauchy-Schwarz inequalities, and the Hsiung--Minkowski formula \eqref{eq-HM}.

\textbf{Equality case: necessity.} If equality in \eqref{eq-ext-c} holds, then \eqref{eq-5.34} is sharp, which implies $\lambda_1=\dots=\lambda_m$ (cf.~Lemma \ref{lem-3.4}).
Moreover, we have $\Delta x=-\lambda_1 x$.
By Takahashi's theorem again, we conclude that $M$ is actually a minimal submanifold in a hypersphere $\mathbb{S}^{n-1}(r)$ with the radius $r=\sqrt{m/\lambda_1}$.

Equalities in \eqref{eq-6.2} and \eqref{eq-6.3} imply $H_T$ is proportional to $x$.
Suppose $H_T=kx$ for some constant $k\neq 0$, then $|H_T|=|k||x|=|k|r$.
From \eqref{eq-2.21} we derive $\tr T=-\l x, H_T\r =-k r^2$ is also constant.

Let $\{e_{m+2}, \cdots, e_{n}\}$ be a normal frame of the immersion from $M$ to $\mathbb{S}^{n-1}(r)$, then $\{e_{m+1}=\frac{x}{r}, e_{m+2}, \cdots, e_{n}\}$ forms a normal frame of $M$ in $\mathbb{R}^n$.
Since $H_T=kx$ and
\begin{equation}\label{eq-6.5}
	H_T=\sum_{i,j,\alpha}T_{ij}h_{ij}^{\alpha}=\sum_{i,j}T_{ij}h_{ij}^{m+1}e_{m+1}+\sum_{i, j; \alpha\ge m+2}T_{ij}h_{ij}^{\alpha}e_{\alpha},
\end{equation}
we conclude that $H_T':=\sum\limits_{i, j; \alpha\ge m+2}T_{ij}h_{ij}^{\alpha}e_{\alpha}$
must vanish, i.e., $M$ is $T$-minimal in $\mathbb{S}^{n-1}(r)$.

\textbf{Equality case: sufficiency.}
Conversely, if $M$ is minimal and $T$-minimal  in $\mathbb{S}^{n-1}(r)$, then $\lambda_1=\dots=\lambda_m$ and \eqref{eq-6.5} becomes
\begin{equation*}
	H_T=\sum_{i,j}T_{ij}h_{ij}^{m+1}e_{m+1}=-\sum_{i,j}T_{ij}\frac{1}{r}\delta_{ij}e_{m+1}=-\tr T \frac{1}{r^2}x,
\end{equation*}
hence, $|H_T|=\frac{1}{r}\tr T$. Since $r^2=m/\lambda_1$ and  $\tr T$ is constant, one can easily check that equality in \eqref{eq-ext-c} holds.
\end{proof}

\begin{proof}[Proof of Theorem \ref{thm-ext-c-sphere}]
	Since $x: M\to \mathbb{S}^n\subset \mathbb{R}^{n+1}$, we treat $M$ as a submanifold of $\mathbb{R}^{n+1}$.
	Let $H_T$ and $H_T'$ are associated to $M\to \mathbb{S}^n$ and $M\to \mathbb{R}^{n+1}$, respectively. Then \eqref{eq-ext-c-sphere} is immediately obtained from Theorem \ref{thm-ext-c}, just noticing the following relations:
	\begin{equation*}
		H_T'=H_T-(\tr T) x,\quad  |H_T'|^2=|H_T|^2+(\tr T)^2.
	\end{equation*} 

\textbf{Equality case:} 
The sufficiency can be verified by using Takahashi's theorem, we omit the details.
Next, we check the necessity.

Since $\tr T\not\equiv 0$ implies $H_T'\not\equiv 0$, equality in \eqref{eq-ext-c-sphere} implies
$\tr T$ is constant, and $M$ is minimal and $T$-minimal in a hypersphere $S'$ of $\mathbb{R}^{n+1}$ by Theorem \ref{thm-ext-c}.

If $S'(r)$ is just $\mathbb{S}^n$, then the immersion is given by a subspace of the first eigenspace. This is the first case.

Otherwise, $M$ lies in $\Sigma_1:=S'\cap \mathbb{S}^n$. We know that $\Sigma_1$ is a geodesic sphere of $\mathbb{S}^n$.
Since $M$ is minimal and $T$-minimal in $S'$, $M$ is also minimal and $T$-minimal in $\Sigma_1$. The geodesic radius of $\Sigma_1$ can be determined by the Gauss equation. This is the second case.
\end{proof}

\begin{proof}[Proof of Theorem \ref{thm-ext-s}]
The proof is similar the proof of Theorem \ref{thm-ext-c},
but there are two main differences: one is that we should apply the Hsiung-Minkowski formula for $\partial M$ instead of $M$ since $\partial M$ is closed;
the other one is that $m=n$ is allowed when considering Steklov eigenvalues.

We assume that the coordinate functions satisfy
\begin{equation*}
	\int_{\partial M}x^A=0 \mbox{ for $A=1,\dots, n$}.
\end{equation*}

We claim
\begin{equation}\label{eq-5.35}
	\mathfrak{H}(\sigma_1,\cdots,\sigma_m) \int_{\partial M} |x|^2 \le m V(M).
\end{equation}

When $n>m$, we assume further
\begin{equation*}
	\int_{{\partial M}}x^Av_B= 0 \mbox{ for $1\le B<A \le n$}
\end{equation*}
by using the QR-decomposition.
Then \eqref{eq-5.35} is obtained by \eqref{eq-5.10}, just setting $f\equiv 1$ and replacing $V(\partial M)=\int_{\partial M}1$ by $\int_{\partial M} |x|^2$.

When $n=m$, this means $M$ is a domain of $\mathbb{R}^m$.
We use the following characterization for the inverse trace of the Steklov eigenvalues (see \cite{Ban80}*{pp.~99}):

\begin{align*}
	\sum_{i=1}^m\frac{1}{\sigma_i}={}
	&\max\left\{
		\sum_{i=1}^{m}\int_{\partial M}u_i^2\right|u_i\in W^{1,2}(M), \int_{M} \nabla u_i\nabla u_j=\delta_{ij}, \int_{M}\nabla u_i \nabla v_0=0, \\
		&{}\qquad \quad \left.\int_{\partial M}u_i=0 \mbox{ for $1\le i,j\le m$} \right\}.
\end{align*}
Noticing that the eigenfunction $v_0$ is constant, we take $u_i=x^i\big/\sqrt{V(M)}$, then
\begin{equation*}
	\sum_{i=1}^m\frac{1}{\sigma_i}\ge \sum_{i=1}^m \int_{\partial M} \Bigl(\frac{x^i}{\sqrt{V(M)}}\Bigr)^2=\frac{1}{V(M)}\int_{\partial M} |x|^2,
\end{equation*}
which is equivalent to \eqref{eq-5.35}.

Multiplying the both sides of \eqref{eq-5.35} by $\int_{\partial M}|H_T|^2$, we have
\begin{align}
	m V(M)	\int_{\partial M}|H_T|^2& \ge \mathfrak{H}(\sigma_1,\cdots,\sigma_m) \int_{\partial M} |x|^2\int_{\partial M}|H_T|^2\label{eq-5.36}\\
	&\ge \mathfrak{H}(\sigma_1,\cdots,\sigma_m)\Bigl(\int_{\partial M}|x|\cdot |H_T| \Bigr)^2\label{eq-5.37}\\
	&\ge \mathfrak{H}(\sigma_1,\cdots,\sigma_m)\Bigl(\int_{\partial M}\l x, H_T\r \Bigr)^2\label{eq-5.38}\\
	&=  \mathfrak{H}(\sigma_1,\cdots,\sigma_m)\Bigl(\int_{\partial M}\tr T \Bigr)^2.\nonumber
\end{align}

\textbf{Equality case: necessity.} Equality in \eqref{eq-ext-s} implies $H_T=kx$ on $\partial M$ for some constant $k\neq 0$ and each coordinate function $x^A$ is an eigenfunction. Since $H_T$ is normal to $\partial M$, so is $x$.
This implies both $|x|$ and $|H_T|$ are constant on any connected components of $\partial M$.
Hence, $x$ maps each connected component of $\partial M$ into a hypersphere of $\mathbb{R}^n$.

Since the coordinate functions are eigenfunctions, they are harmonic, and then $x(M)$ is minimal (cf.~\eqref{eq_d2x}).

Analogous to Lemma \ref{lem-5}, we also have $\sigma_1=\dots=\sigma_m$.
This means $x(\partial M)$ lies in the same hypersphere of radius $\frac{1}{\sigma_1}$.

Since $\Delta |x|^2= 2m$ in $M$ and $|x|^2= 1/\sigma_1^2$ on $\partial M$,
we have $|x|^2\le 1/\sigma_1^2$ by the maximum principle.
Hence, $x(M)\subset B^n(1/\sigma_1)$.

When $n>m$, $H_T=kx$ on $\partial M$ implies that $\partial M$ is $T$-minimal in $\partial B^n(\frac{1}{\sigma_1})$.
We also have that $x(M)$ meets $\partial B^n(\frac{1}{\sigma_1})$ orthogonally along the boundary.

When $n=m$, we derive $M=B^m(\frac{1}{\sigma_1})$. It follows from \eqref{eq-2.21} that $\tr T=-\langle x, H_T\rangle=-k|x|^2=-k/\sigma_1^2$ on $\partial M$, which is constant.

\textbf{Equality case: sufficiency.}
Conversely, when $n>m$, if $x$ immerses $M$ minimally in $B^n(\frac{1}{\sigma_1})$ with $x(\partial M)\subset \partial B^n(\frac{1}{\sigma_1})$ such that $x(M)$ meets $\partial B^n(\frac{1}{\sigma_1})$ orthogonally along $x(\partial M)$, then each coordinate function is the Steklov eigenfunction associated to $\sigma_1$, so \eqref{eq-5.35} (i.e., \eqref{eq-5.36}) becomes equality.
Additionally, if $\partial M$ is $T$-minimal in $\partial B^n(\frac{1}{\sigma_1})$,
then $H_T$ is parallel to $x$, so both \eqref{eq-5.37} and \eqref{eq-5.38} become equalities.
Hence, equality in \eqref{eq-ext-s} holds.

When $n=m$, we know that $\sigma_1=\dots=\sigma_m=1/r$ for a ball $B^m(r)$.
On the boundary $\partial B^m(r)$ we have $H_T=-\tr T \frac{1}{r}x$ on.
If $\tr T$ is constant, then $|H_T|^2=|\tr T|^2$ is also constant.
Hence, equality in \eqref{eq-ext-s} holds due to the relation $mV(B^m(r))=r V(\partial B^m(r))$.
\end{proof}

\textbf{Acknowledgment:}
The author was partially supported by NSFC (Grant No.~12571054), Natural Science Foundation of Shaanxi Province (Grant No.~2024JC-YBMS-011) and Shaanxi Fundamental Science Research Project for Mathematics and Physics (Grant No.~22JSQ005).






\end{document}